\documentclass[12pt,reqno]{amsart}

\newtheorem{thm}{Theorem}[section]
\newtheorem{defi}{Definition}[section]

\newtheorem{cor}{Corollary}[section]

\theoremstyle{definition}
\newtheorem{rem}{Remark}[section]

\newcommand{\be}{\begin{equation}}
\newcommand{\ee}{\end{equation}}
\newcommand{\bea}{\begin{eqnarray}}
\newcommand{\eea}{\end{eqnarray}}
\newcommand{\beb}{\begin{eqnarray*}}
\newcommand{\eeb}{\end{eqnarray*}}
\usepackage{amssymb,amsfonts,amsthm,setspace,indentfirst,mathrsfs}
\usepackage{minibox}
\usepackage{enumitem}
\numberwithin{equation}{section}
\begin{document}
%
\title[Curvature properties of the Vaidya metric]{\bf{Curvature properties of the Vaidya metric}}
\author[A. A. Shaikh, H. Kundu and J. Sen]{Absos Ali Shaikh$^1$, Haradhan Kundu$^1$ and Jayashree Sen$^1$}
\date{\today}
\address{A. A. Shaikh, H. Kundu and J. Sen,
\newline Department of Mathematics,
\newline University of Burdwan, Golapbag,
\newline Burdwan-713104,
\newline West Bengal, India}
\email{aask2003@yahoo.co.in, aashaikh@math.buruniv.ac.in}
\email{kundu.haradhan@gmail.com}
\email{jayashreesen92@gmail.com}
\dedicatory{Dedicated to Professor Prahalad Chunnilal Vaidya on his 100$^{\mbox{th}}$ birthday}
\begin{abstract}
As a generalization of the Schwarzschild solution, Vaidya presented a radiating metric to develop a model of the exterior of a star including its radiation field, named later Vaidya metric. The present paper deals with the investigation on the curvature properties of Vaidya metric. It is shown that Vaidya metric can be considered as a model of different pseudosymmetric type curvature conditions, namely, $C\cdot C = \frac{m}{r^3} Q(g, C)$, $R\cdot R - Q(S,R) = \frac{m}{r^3} Q(g,C)$ etc. It is also shown that Vaidya metric is Ricci simple, vanishing scalar curvature and its Ricci tensor is Riemann-compatible. As a special case of the main result, we obtain the curvature properties of Schwarzschild metric. Finally, we compare the curvature properties of Vaidya metric with another radiating metric, namely, Ludwig-Edgar pure radiation metric.
\end{abstract}
%
\textbf{Published in:} Indian Journal of Mathematics, Vol. 61, No. 1, 2019, 41-59\\

\maketitle
%

\section{\bf Introduction}\label{intro}
Let $M$ be a semi-Riemannian manifold of dimension $n$ endowed with a semi-Riemannian metric $g$ and let $\nabla$, $R$, $S$ and $\kappa $ be respectively the Levi-Civita connection, the Riemann-Christoffel curvature tensor, the Ricci tensor and the scalar curvature of $M$. The additional restriction(s) on the curvature tensor(s) of a specific manifold help us to realize the geometry of that manifold. For example, if on a manifold $R = 0$ (resp., $\nabla R = 0$) then it is flat (resp., locally symmetric). Hence it is very important to investigate the curvature restricted geometric structures on a manifold.\\
\indent In the literature of differential geometry we find various generalizations of locally symmetric manifolds (\cite{Cart26}, \cite{Cart27}) in several directions, such as recurrent manifold by Ruse (\cite{Ru46}, \cite{Ru49}, \cite{Ru49a}, \cite{Wa50}), some generalized recurrent manifolds by Shaikh and his coauthors (\cite{Du79}, \cite{SR10}, \cite{SP10}, \cite{SR11}, \cite{SAR13}, \cite{SRK16}), semisymmetric manifolds by Cartan \cite{Cart46} and Sinyukov (\cite{Siny54}, \cite{Siny79}) (which were latter classified by Szab$\acute{\mbox{o}}$ (\cite{Szab82}, \cite{Szab84}, \cite{Szab85})), generalized semisymmetric manifold by Mike\v{s} (\cite{MIKES76}, \cite{Mike88}, \cite{Mike92}, \cite{MIKES96}, \cite{MSV15}, \cite{MVH09}), pseudosymmetric manifolds by Chaki \cite{Chak87}, pseudosymmetric manifolds by Deszcz (\cite{AD83}, \cite{Desz92}, \cite{DGr87}, \cite{DHV08}, \cite{HV07}, \cite{HV09}, \cite{LV94}, \cite{LV16}, \cite{LV17}), weakly symmetric manifolds by Tam$\acute{\mbox{a}}$ssy and Binh (\cite{TB89}, \cite{TB93}), manifolds of recurrent curvature 2-forms (\cite{Bess87}, \cite{LR89}) etc.\\
\indent In his theory of general relativity, Einstein made a bridge between the geometrical and physical quantities of a spacetime (a connected 4-dimensional Lorentzian manifold) by presenting the famous field equation
$$S - \frac{\kappa}{2}g + \Lambda g = \frac{8 \pi G}{c^4} T,$$
where $\Lambda$ is the cosmological constant, $G$  is Newton's gravitational constant, $c$ is the speed of light in vacuum and $T$ is the energy-momentum tensor. For instance, a spacetime represent perfect fluid if and only if it is quasi-Einstein, i.e., its Ricci tensor satisfies the condition $S = \alpha g + \beta \Pi\otimes\Pi$, where $\alpha, \beta$ are some scalars and $\Pi$ is a 1-form (\cite{Chen17}, \cite{Chen17-KJM}, \cite{DGHS98}, \cite{SKH11}). Therefore the study of curvature restricted geometric structures become essential both physically and geometrically in the investigation of a spacetime.\\
\indent According to Birkhoff's theorem \cite{Birk23}, the vacuum spherically symmetric spacetime with zero cosmological constant is the asymptotically flat solution given by Schwarzschild \cite{GP09}
\be\label{sm}
ds^2=-\left(1-\frac{2m}{r}\right)du^2-2dr du + r^2 (d\theta^2+\sin^2\theta d\phi^2),
\ee
where $m$ is an arbitrary parameter. To develop a model of the exterior of a star, which includes its radiation field, there arose an important generalization of the Schwarzschild solution by Vaidya \cite{Vaid43}. In terms of $(r, t, \theta, \phi)$-coordinates Vaidya metric \cite{Vaid43} can be written as
$$
ds^2=\left(1-\frac{2m}{r}\right)^{-1}\left(dr^2 - \frac{(m_t)^2}{(m_r)^2} dt^2\right) + r^2(d\theta^2+\sin^2\theta d\phi^2),
$$
where $m = m(r, t)$. In fact, the above metric can be expressed in a much more useful form (for outgoing radiation) by introducing a null coordinate $u$, as follows:
\be\label{vm}
ds^2=-\left(1-\frac{2m(u)}{r}\right)du^2 - 2drdu + r^2(d\theta^2+\sin^2\theta d\phi^2).
\ee
If $m(u)=constant$, then the Vaidya metric \eqref{vm} reduces to the Schwarzschild metric \eqref{sm}. Hence \eqref{vm} is a generalization of the Schwarzschild metric \eqref{sm}. In this case, $m(u)$ is an arbitrary non-increasing function of the retarded null coordinate $u$.\\
\indent It is known that the Schwarzschild metric is a pseudosymmetric metric \cite{1991-DVV} (Proposition 2). The Schwarzschild metric, the Kottler metric, the Reissner-Nordstr\"{o}m metric, as well as some Friedmann-Lema\^{i}tre-Robertson-Walker metrics are the ``oldest'' examples of pseudosymmetric non-semisymmetric warped product metrics (see, e.g., \cite{DGJZ-2016}, Section 1, \cite{DHV08}, Chapter 6). In \cite{Kowa06} (Section 5) it was stated that the Vaidya metric $g$ is a non-pseudosymmetric metric, rank $S \leq 1$ and the tensors $R \cdot R - Q(S,R)$ and $Q(g,C)$ are linearly dependent. Further, in view of Theorem 2 of \cite{D-1991}, the tensors $C \cdot C$ and $Q(g,C)$ of the metric $g$ are linearly dependent. These properties of $g$, in details, are presented in Example 7.1 (ii)-(iii) of \cite{DGJZ-2016} (eq. (7.14)). In addition, in that example it was mention that $g$ satisfies also the following conditions: $S^2 = 0$, $\kappa= 0$ and $S \cdot R = 2mr^{-3} g \wedge S$.\\
\indent The main object of the present paper is to investigate the geometric structures admitted by the Vaidya metric \eqref{vm}. It is interesting to note that such a metric admits several important geometric structures, such as, it is Ricci simple, manifold of vanishing scalar curvature, its Ricci tensor is Riemann-compatible and its conformal curvature 2-forms are recurrent. Moreover Vaidya metric satisfies various pseudosymmetric type curvature conditions, such as pseudosymmetric Weyl conformal curvature tensor, the difference tensor $R\cdot R - Q(S,R)$ is linearly dependent with $Q(g,C)$. As a special case of the main result, we obtain the curvature properties of Schwarzschild metric. It is shown that Schwarzschild metric is pseudosymmetric in the sense of Deszcz.\\
\indent The paper is organized as follows. Section 2 deals with the different geometric structures due to various curvature tensors. Section 3 is concerned with the calculations of components of various tensors of Vaidya spacetime. Section 4 is devoted to the conclusion on the geometric structures admitted by Vaidya metric and Schwarzschild metric. Finally, in section 5, we compare the curvature properties of Vaidya metric with Ludwig-Edgar pure radiation metric.
\section{\bf Curvature Restricted Geometric Structures}
A curvature restricted geometric structure on a semi-Riemannian manifold $M$ is a geometric structure obtained by imposing some restriction(s) on some curvature tensor(s) of $M$ by means of covariant derivatives of first order or higher orders. In this section we will explain some useful notations and definitions of various curvature restricted geometric structures.\\
\indent Let $A$ and $E$ be two symmetric $(0,2)$-tensors. The Kulkarni-Nomizu product $A\wedge E$ of $A$ and $E$ is defined by (see e.g. \cite{DGHS11}, \cite{Glog02})
\begin{eqnarray*}
(A\wedge E)(X_1,X_2,X_3,X_4) &=& A(X_1,X_4)E(X_2,X_3) + A(X_2,X_3)E(X_1,X_4)\\
&-& A(X_1,X_3)E(X_2,X_4) - A(X_2,X_4)E(X_1,X_3),
\end{eqnarray*}
where $X_1,X_2,X_3,X_4 \in \chi(M)$, the Lie algebra of all smooth vector fields on $M$. Throughout the paper we will consider $X, Y, X_1, X_2, \cdots \in \chi(M)$. Again for a symmetric $(0,2)$-tensor $E$, we can define two endomorphisms $\mathcal E$ and $(X\wedge_E Y)$ as (\cite{DDHKS00}, \cite{DGHS11})
$$g(\mathcal E X_1, X_2) = E(X_1, X_2) \ \mbox{ and } \ (X\wedge_E Y)X_1 = E(Y,X_1)X-E(X,X_1)Y.$$
Now in terms of Kulkarni-Nomizu product $\wedge$ and the endomorphism $(X\wedge_E Y)$, the Gaussian curvature tensor $\mathfrak G$, the projective curvature tensor $P$, the conharmonic curvature tensor $K$, the concircular curvature tensor $W$ and the conformal curvature tensor $C$ are given by (\cite{Ishi57}, \cite{YK89})
$$\mathfrak G = \wedge_g, \ \ \ \ P = R - \frac{1}{n-1}(\wedge_S), \ \ \ \ K = R-\frac{1}{n-2}(g\wedge S),$$
$$W = R-\frac{r}{n(n-1)}\mathfrak G, \ \ \ \ C = K + \frac{r}{(n-2)(n-1)}\mathfrak G.$$
\indent For a $(0,4)$-tensor $D$ we can define the corresponding $(1,3)$-tensor $\mathcal D$ and the endomorphism $\mathscr{D}(X,Y)$ as follows:
$$D(X_1,X_2,X_3,X_4) = g(\mathcal D(X_1,X_2)X_3, X_4), \ \ \mathscr{D}(X,Y)X_1 = \mathcal D(X,Y)X_1.$$
\indent Now operating $\mathscr{D}(X,Y)$ and $(X\wedge_A Y)$ on a $(0,k)$-tensor $H$, $k\geq 1$, we get two $(0,k+2)$-tensors $D\cdot H$ and $Q(A,H)$ respectively as follows (see \cite{DG02}, \cite{DGHS98}, \cite{DH03}, \cite{SDHJK15}, \cite{SK14}, \cite{Tach74} and also references therein)
$$D\cdot H(X_1,X_2,\cdots,X_k,X,Y) = -H(\mathcal D(X,Y)X_1,X_2,\cdots,X_k) - \cdots - H(X_1,X_2,\cdots,\mathcal D(X,Y)X_k).$$
and
\beb
&&Q(A,H)(X_1,X_2, \ldots ,X_k,X,Y) = ((X \wedge_A Y)\cdot H)(X_1,X_2, \ldots ,X_k)\\
&&= A(X, X_1) H(Y,X_2,\cdots,X_k) + \cdots + A(X, X_k) H(X_1,X_2,\cdots,Y)\\
&& - A(Y, X_1) H(X,X_2,\cdots,X_k) - \cdots - A(Y, X_k) H(X_1,X_2,\cdots,X).
\eeb
\begin{defi}$($\cite{AD83}, \cite{Cart46}, \cite{Desz92}, \cite{HV04}, \cite{SK14}, \cite{SKppsn}, \cite{SKppsnw}, \cite{SRK15}, \cite{Szab82}$)$
A semi-Riemannian manifold $M$ is said to be $H$-semisymmetric type if $D\cdot H = 0$ and it is said to be $H$- pseudosymmetric type if $\left(\sum\limits_{i=1}^k c_i D_i\right)\cdot H = 0$ for some scalars $c_i$'s, where $D$ and each $D_i$, $i=1,\ldots, k$, $(k\ge 2)$, are (0,4) curvature tensors.
\end{defi}
\indent In particular, if $D = R$ and $H=R$ (resp., $S$, $C$, $W$ and $K$), then $M$ is called semisymmetric (resp., Ricci, conformally, concircularly and conharmonically semisymmetric). Again, if $i =2$, $D_1 = R$, $D_2 = \mathfrak G$ and $H= R$ (resp., $S$, $C$, $W$ and $K$), then $M$ is called Deszcz pseudosymmetric (resp., Ricci, conformally, concircularly and conharmonically pseudosymmetric). Especially, if $i =2$, $D_1 = C$, $D_2 = \mathfrak G$ and $H =C$, then $M$ is called a manifold of pseudosymmetric Weyl conformal curvature tensor.\\
\indent We note that the pp-wave metric \cite{SBK17} is semisymmetric and the Robinson-Trautman metric \cite{SAA17} is pseudosymmetric.
\begin{defi} 
A semi-Riemannian manifold $M$ is said to be quasi-Einstein if $S = \alpha g + \beta \Pi \otimes \Pi$ for some scalars $\alpha$ and $\beta$, and a 1-form $\Pi$. In particular, if $\beta\equiv 0$ (resp., $\alpha\equiv 0$), then a quasi-Einstein manifold is called Einstein \cite{Bess87} (resp., Ricci simple \cite{Chen17}).
\end{defi}
We mention that the Robertson-Walker spacetimes are quasi Einstein \cite{ADEHM14}, Kaigorodov metric \cite{Kaig63} is Einstein, and G\"{o}del spacetime \cite{DHJKS14} is Ricci simple. Quasi-Einstein manifolds were investigated in several papers, see, e.g., \cite{DDHKS00}, \cite{DGHS98}, \cite{DGHS11}, \cite{DGJZ-2016}, \cite{DH03}, \cite{SDHJK15}.\\
\indent In 1983, Derdzinski and Shen \cite{DS83} showed that if $E$ is a Codazzi tensor \cite{Gray78} such that $V_{\lambda}$ and $V_{\mu}$ are two eigenspaces corresponding to the eigenvalues $\lambda$ and $\mu$ of the operator $\mathcal E$, then the subspace $V_{\lambda} \wedge V_{\mu}$ is invariant under the action of the curvature operator $\mathscr R(X,Y)$. Recently, Mantica and Molinari \cite{MM12a} extended  their theorem by replacing the condition `Codazzi type' by `Riemann-compatibility'.
\begin{defi}\label{def2.8}
Let $D$ be a $(0,4)$-tensor and $E$ be a symmetric $(0, 2)$-tensor on $M$. Then $E$ is said to be $D$-compatible (\cite{DGJPZ13}, \cite{MM12b}, \cite{MM13}) if
\[
D(\mathcal E X_1, X,X_2,X_3) + D(\mathcal E X_2, X,X_3,X_1) + D(\mathcal E X_3, X,X_1,X_2) = 0
\]
holds, where $\mathcal E$ is the endomorphism corresponding to $E$. Again a 1-form $\Pi$ is said to be $D$-compatible if $\Pi\otimes \Pi$ is $D$-compatible.
\end{defi}
We note that the Ricci tensor of G\"{o}del metric \cite{DHJKS14} and Som-Raychaudhuri metric \cite{SK16srs} are Riemann-compatible.\\
\indent The curvature 2-forms $\Omega_{(D)l}^m$ (\cite{Bess87}, \cite{LR89}) associated to a  $(0,4)$-curvature tensor $D$ and the 1-forms $\Lambda_{(Z)l}$ \cite{SKP03} associated to a symmetric (0, 2)-curvature tensor $Z$ are defined as follows:
$$\Omega_{(D)l}^m= D_{jkl}^m dx^j \wedge dx^k \mbox{ and } \Lambda_{(Z)l} = Z_{lm} dx^m,$$
where $\wedge$ indicates the exterior product. Now $\Omega_{(D)l}^m$ (resp., $\Lambda_{(Z)l}$) is recurrent if
$$\mathfrak D \Omega_{(D)l}^m  = \Pi \wedge \Omega_{(D)l}^m \mbox{ (resp., } \mathfrak D \Lambda_{(Z)l}  = \Pi \wedge \Lambda_{(Z)l}),$$
where $\mathfrak D$ is the exterior derivative and $\Pi$ is the associated $1$-form. Recently Mantica and Suh (\cite{MS12a}, \cite{MS13a}, \cite{MS14}) showed that $\Omega_{(D)l}^m$ are recurrent if and only if
\bea\label{man}
&&(\nabla_{X_1} D)(X_2,X_3,X,Y)+(\nabla_{X_2} D)(X_3,X_1,X,Y)+(\nabla_{X_3} D)(X_1,X_2,X,Y) =\\\nonumber
&&\hspace{1in} \Pi(X_1) D(X_2,X_3,X,Y) + \Pi(X_2) D(X_3,X_1,X,Y)+ \Pi(X_3) D(X_1,X_2,X,Y)
\eea
and $\Lambda_{(Z)l}$ are recurrent if and only if
\be\label{mans}
(\nabla_{X_1} Z)(X_2,X) - (\nabla_{X_2} Z)(X_1,X) = \Pi(X_1) Z(X_2,X) - \Pi(X_2) Z(X_1,X)
\ee
for a 1-form $\Pi$.
\section{\bf Various tensors of Vaidya metric}\label{com}
In terms of $(u, r, \theta, \phi)$-coordinates ($r>0$), the Vaidya metric \eqref{vm} is given by
\beb
g = \left(\begin{array}{cccc}
 -1+\frac{2m}{r} & -1 & 0 & 0 \\
 -1 & 0 & 0 & 0 \\
 0 & 0 & r^2 & 0 \\
 0 & 0 & 0 & r^2 \sin^2\theta
\end{array}\right).
\eeb
Then the non-zero components of its Riemann-Christoffel curvature tensor $R$, Ricci tensor $S$ and scalar curvature $\kappa$ of \eqref{vm} are given by
\begin{equation}\label{R}
\left.\begin{array}{c}
R_{1212}=-\frac{2 m}{r^3}, \ \ R_{1313}=-\frac{2 m^2+r^2 m'-m r}{r^2}, \ \ R_{1424}=\frac{m \sin ^2 \theta}{r},\\
R_{1323}=\frac{m}{r}, \ \ R_{1414}=-\frac{\left(2 m^2+r^2 m'-m r\right) \sin ^2 \theta}{r^2}, \ \ R_{3434}=2 m r \sin ^2 \theta,
\end{array}\right\}
\end{equation}

\begin{equation}\label{S} S_{11}=\frac{2 m'}{r^2} \ \mbox{and } \kappa = 0, \end{equation}
where $m= m(u)$ and $m' = \frac{dm}{du}$.
The non-zero components of its conformal curvature tensor $C$ are given by
\begin{equation}\label{C}
\left.\begin{array}{c}
C_{1212}=-\frac{2 m}{r^3}, \ \ C_{1313}=-\frac{m (2 m-r)}{r^2}, \ \ C_{1323}=\frac{m}{r},\\
C_{1414}=-\frac{m (2 m-r) \sin ^2 \theta}{r^2}, \ \ C_{1424}=\frac{m \sin ^2 \theta}{r}, \ \ C_{3434}=2 m r \sin ^2 \theta.
\end{array}\right\}
\end{equation}
now the non-zero components of its $\nabla R$, $\nabla S$ and $\nabla C$ of \eqref{vm} are given by
\begin{equation}\label{dR}
\left.\begin{array}{c}
R_{1212,1}=-\frac{2 m'}{r^3}, \ \ R_{1212,2}=\frac{6 m}{r^4}, \ \ R_{1213,3}=\frac{6 m^2+r^2 m'-3 m r}{r^3},\\
R_{1214,4}=\frac{\sin ^2(\theta ) \left(6 m^2+r^2 m'-3 m r\right)}{r^3},\ \ R_{1223,3}= R_{1323,2}=-\frac{3 m}{r^2},\\
R_{1224,4}= R_{1424,2}=-\frac{3 m \sin ^2(\theta )}{r^2},\ \ R_{1313,1}=-\frac{r^2 m''-r m'+4 m m'}{r^2},\\
R_{1313,2}=\frac{6 m^2+2 r^2 m'-3 m r}{r^3}, \ \ R_{1323,1}=\frac{m'}{r}, \ \ R_{1414,1}=-\frac{\sin ^2(\theta ) \left(r^2 m''-r m'+4 m m'\right)}{r^2},\\
-R_{1334,4}=r \sin ^2(\theta ) m', \ \ R_{1434,3}=r \sin ^2(\theta ) m', \ \ \frac{1}{2}R_{3434,1}=r \sin ^2(\theta ) m',\\
R_{1414,2}=\frac{\sin ^2(\theta ) \left(6 m^2+2 r^2 m'-3 m r\right)}{r^3}, \ \ R_{1424,1}=\frac{\sin ^2(\theta ) m'}{r},\\
R_{2334,4}= -R_{2434,3}= -\frac{1}{2}R_{3434,2}=3 m \sin ^2(\theta ),
\end{array}\right\}
\end{equation}
\begin{equation}\label{dS}
\left.\begin{array}{c}
S_{11,1}=\frac{2 \left(r^2 m''+2 m m'\right)}{r^4}, \ \ S_{11,2}=-\frac{4 m'}{r^3},\\
S_{13,3}=-\frac{2 m'}{r},\ \ S_{14,4}=-\frac{2 \sin ^2(\theta) m'}{r},
\end{array}\right\}
\end{equation}
\begin{equation}\label{dC}
\left.\begin{array}{c}
C_{1212,1}=-\frac{2 m'}{r^3}, \ \ C_{1212,2}=\frac{6 m}{r^4}, \ \ C_{1213,3}= C_{1313,2}=\frac{3 m (2 m-r)}{r^3},\\
C_{1214,4}= C_{1414,2}=\frac{3 m \sin ^2(\theta ) (2 m-r)}{r^3},\\
C_{1223,3}= C_{1323,2}=-\frac{3 m}{r^2}, \ \ C_{1224,4}= C_{1424,2}=-\frac{3 m \sin ^2(\theta )}{r^2},\\
C_{1313,1}=\frac{(r-2 m) m'}{r^2}, \ \ C_{1323,1}=\frac{m'}{r}, \ \ C_{1414,1}=\frac{\sin ^2(\theta ) (r-2 m) m'}{r^2},\\
C_{1424,1}=\frac{\sin ^2(\theta ) m'}{r}, \ \ C_{2334,4}= -C_{2434,3}= -\frac{1}{2}C_{3434,2}=3 m \sin ^2(\theta ), \ \ C_{3434,1}=2 r \sin ^2(\theta ) m'.
\end{array}\right\}
\end{equation}
Now from above we can easily evaluate the non-zero components of $R\cdot R$, $C\cdot R$, $R\cdot C$, $C\cdot C$, $Q(g,R)$, $Q(S,R)$, $Q(g,C)$ and $Q(S, C)$ as follows:
\begin{equation}\label{R.R}
\left.\begin{array}{c}
-2R\cdot R_{121313}= R\cdot R_{131312}=\frac{4 m m'}{r^3}, \ \ R\cdot R_{121323}= -R\cdot R_{122313}=\frac{3 m^2}{r^4},\\
-2R\cdot R_{121414}= R\cdot R_{141412}=\frac{4 m m' \sin^2\theta}{r^3}, \ \ R\cdot R_{121424}= -R\cdot R_{122414}=\frac{3 m^2 \sin^2\theta}{r^4},\\
-R\cdot R_{133414}= R\cdot R_{143413}=\frac{m \left(6 m^2+4 r^2 m'-3 m r\right) \sin^2\theta}{r^3},\\
R\cdot R_{133424}= -R\cdot R_{143423}= R\cdot R_{233414}= -R\cdot R_{243413}=\frac{3 m^2 \sin^2\theta}{r^2},
\end{array}\right\}
\end{equation}
\indent\\
\begin{equation}\label{R.C}
\left.\begin{array}{c}
R\cdot C_{121313}=-\frac{3 m m'}{r^3}, \ \ R\cdot C_{121323}= -R\cdot C_{122313}=\frac{3 m^2}{r^4},\\
R\cdot C_{121414}=-\frac{3 m m' \sin^2\theta}{r^3}, \ \ R\cdot C_{121424}= -R\cdot C_{122414}=\frac{3 m^2 \sin^2\theta}{r^4},\\
-R\cdot C_{133414}= R\cdot C_{143413}=\frac{3 m \left(2 m^2+r^2 m'-m r\right) \sin^2\theta}{r^3},\\
R\cdot C_{133424}= -R\cdot C_{143423}= R\cdot C_{233414}= -R\cdot C_{243413}=\frac{3 m^2 \sin^2\theta}{r^2},
\end{array}\right\}
\end{equation}
\indent\\
\begin{equation}\label{C.R}
\left.\begin{array}{c}
4C\cdot R_{121313}= C\cdot R_{131312}=\frac{4 m m'}{r^3}, \ \ C\cdot R_{121323}= -C\cdot R_{122313}=\frac{3 m^2}{r^4},\\
4C\cdot R_{121414}= C\cdot R_{141412}=\frac{4 m m' \sin^2\theta}{r^3}, \ \ C\cdot R_{121424}= -C\cdot R_{122414}=\frac{3 m^2 \sin^2\theta}{r^4},\\
-C\cdot R_{133414}= C\cdot R_{143413}=\frac{m \left(6 m^2+r^2 m'-3 m r\right) \sin^2\theta}{r^3},\\
$$C\cdot R_{133424}= -C\cdot R_{143423}= C\cdot R_{233414}= -C\cdot R_{243413}=\frac{3 m^2 \sin^2\theta}{r^2},
\end{array}\right\}
\end{equation}
\indent\\
\begin{equation}\label{C.C}
\left.\begin{array}{c}
C\cdot C_{121323}= -C\cdot C_{122313}=\frac{3 m^2}{r^4}, \ \ C\cdot C_{121424}= -C\cdot C_{122414}=\frac{3 m^2 \sin^2\theta}{r^4},\\
-C\cdot C_{133414}= C\cdot C_{143413}=\frac{3 m^2 (2 m-r) \sin^2\theta}{r^3},\\
C\cdot C_{133424}= -C\cdot C_{143423}= C\cdot C_{233414}= -C\cdot C_{243413}=\frac{3 m^2 \sin^2\theta}{r^2},
\end{array}\right\}
\end{equation}
\indent\\
\begin{equation}\label{qgr}
\left.\begin{array}{c}
2Q(g,R)_{121313}= -Q(g,R)_{131312}=2 m', \ \ Q(g,R)_{121323}= -Q(g,R)_{122313}=\frac{3 m}{r},\\
2Q(g,R)_{121414}= -Q(g,R)_{141412}=2 m' \sin^2\theta, \ \ Q(g,R)_{121424}= -Q(g,R)_{122414}=\frac{3 m \sin^2\theta}{r},\\
-Q(g,R)_{133414}= Q(g,R)_{143413}= \left(6 m^2+r^2 m'-3 m r\right) \sin^2\theta,\\
Q(g,R)_{133424}= -Q(g,R)_{143423}= Q(g,R)_{233414}= -Q(g,R)_{243413}=3 m r \sin^2\theta,
\end{array}\right\}
\end{equation}
\indent\\
\begin{equation}\label{qsr}
\left.\begin{array}{c}
-2Q(S, R)_{121313}= Q(S, R)_{131312}=\frac{4 m m'}{r^3}, \ \ -2Q(S, R)_{121414}= Q(S, R)_{141412}=\frac{4 m m' \sin^2\theta}{r^3},\\
$$-Q(S, R)_{133414}= Q(S, R)_{143413}=\frac{4 m m' \sin^2\theta}{r},
\end{array}\right\}
\end{equation}
\indent\\
\begin{equation}\label{qgc}
\left.\begin{array}{c}
Q(g, C)_{121323}= -Q(g, C)_{122313}=\frac{3 m}{r}, \ \ Q(g, C)_{121424}= -Q(g, C)_{122414}=\frac{3 m \sin^2\theta}{r},\\
-Q(g, C)_{133414}= Q(g, C)_{143413}=3 m (2 m-r) \sin^2\theta,\\
Q(g, C)_{133424}= -Q(g, C)_{143423}= Q(g, C)_{233414}= -Q(g, C)_{243413}=3 m r \sin^2\theta,
\end{array}\right\}
\end{equation}
\indent\\
\begin{equation}\label{qsc}
\left.\begin{array}{c}
-2Q(S, C)_{121313}= Q(S, C)_{131312}=\frac{4 m m'}{r^3}, \ \ -2Q(S, C)_{121414}= Q(S, C)_{141412}=\frac{4 m m' \sin^2\theta}{r^3},\\
-Q(S, C)_{133414}= Q(S, C)_{143413}=\frac{4 m m' \sin^2\theta}{r}.
\end{array}\right\}
\end{equation}

Now from Einstein's field equations with zero cosmological constant, the energy momentum tensor is given by
$$T= \frac{c^4}{8\pi G}\left(S- \frac{\kappa}{2} g\right),$$ 
where $c=$ speed of light in vacuum, $G=$ gravitational constant. Then the only non-zero component (upto symmetry) of $T$ is given by
\be\label{emt}
T_{11}= \frac{2 c^4 m'}{8 \pi  G r^2}.
\ee
Then the non-zero components of $\nabla T$ are given by
\begin{equation}\label{delt}
\left.\begin{array}{c}
T_{11,1}=\frac{c^4 \left(r^2 m''+2 m m'\right)}{4 \pi  G r^4}, \ \ T_{13,3}=-\frac{c^4 m'}{4 \pi  G r},\\
T_{11,2}=-\frac{c^4 m'}{2 \pi  G r^3}, \ \ T_{14,4}=-\frac{c^4 \sin ^2(\theta ) m'}{4 \pi  G r}.
\end{array}\right\}
\end{equation}
\section{\bf Curvature properties of Vaidya metric}\label{con}
\indent From above we can conclude that the metric \eqref{vm} fulfills the following curvature restricted geometric structures.
\begin{thm}\label{mainthm}
The Vaidya metric given in \eqref{vm} possesses the following curvature properties:
\begin{enumerate}[label=(\roman*)]
\item neither Ricci flat nor Ricci symmetric but its scalar curvature is zero and hence $R = W$ and $C = K$.
\item neither conformally semisymmetric nor Deszcz pseudosymmetric but $C\cdot C=\frac{m}{r^3}Q(g,C)$.
\item neither Ricci generalized pseudosymmetric nor Weyl pseudosymmetric but satisfies the pseudosymmetric type conditions $R\cdot R-Q(S,R)=\frac{m}{r^3}Q(g,C)$ and $R\cdot C + C\cdot R = \frac{2 m}{r^3} Q(g, C) + Q(S,C)$.
\item neither conformally recurrent nor the curvature 2-forms $\Omega_{(R)l}^m$ nor the Ricci 1-forms $\Lambda_{(S)l}$ are recurrent but the  conformal 2-forms $\Omega_{(C)l}^m$ are recurrent for the 1-form $\Pi=\left\{\frac{m'}{m},0,0,0\right\}$.
\item not Einstein but it is Ricci simple, such that $S = \beta (\eta\otimes\eta)$, where $\beta = 2m'$ and $\eta =\{\frac{1}{r},0,0,0\}$ with $||\eta|| = 0$. Hence $S\wedge S = 0$ and $S^2 = 0$.
\item Ricci tensor is neither cyclic parallel nor Codazzi type but Riemann compatible as well as Weyl compatible.
\item the general form of the compatible tensors for $R$ and $P$ are given by
$$\left(
\begin{array}{cccc}
 a_{_{11}} & a_{_{12}} & 0 & 0 \\
 a_{_{12}}+\frac{r m'}{m} a_{_{22}} & a_{_{22}} & 0 & 0 \\
 0 & 0 & a_{_{33}} & a_{_{43}} \\
 0 & 0 & a_{_{43}} & a_{_{44}}
\end{array}
\right),$$
where $a_{_{ij}}$ being arbitrary scalars.
\item the general form of the compatible tensors for $C$ and $K$ are given by
$$\left(
\begin{array}{cccc}
 a_{_{11}} & a_{_{12}} & 0 & 0 \\
 a_{_{12}} & a_{_{22}} & 0 & 0 \\
 0 & 0 & a_{_{33}} & a_{_{34}} \\
 0 & 0 & a_{_{34}} & a_{_{44}}
\end{array}
\right).$$
\end{enumerate}
\end{thm}
\begin{proof}
\textit{(i)} From \eqref{S}, (i) is obvious.\\
\textit{(ii)} From \eqref{R.R} and \eqref{qgr}, we get $R\cdot R_{131312} = -\frac{2m}{r^3}Q(g,R)_{131312}$ but $R\cdot R_{122313} = \frac{m}{r^3}Q(g,R)_{122313}$. Hence the Vaidya metric is not pseudosymmetric. Again from \eqref{C.C} and \eqref{qgc}, we get $C\cdot C\ne 0$ but $C\cdot C=\frac{m}{r^3}Q(g,C)$.\\
\textit{(iii)} Again from \eqref{R.R}, \eqref{qsr}, \eqref{R.C}, \eqref{C.R}, \eqref{qgc} and \eqref{qsc}, one can easily check the pseusosymmetric type conditions given in (iii).\\
\textit{(iv)} Now putting the values of $\nabla C$ (given in \eqref{dC}) and $C$ (given in \eqref{C}) and $\Pi=\left\{\frac{m'}{m},0,0,0\right\}$ in \eqref{man} for $D=C$, we get (iv).\\ 
\textit{(v)} It is easily follows from \eqref{S}.\\
\textit{(vi)} From \eqref{S}, \eqref{dS} and definition of compatible tensor, we get (vi).\\
\textit{(vii)} We can easily check (vii) using components of $R$ and $P$.\\
\textit{(viii)} Again (viii) follows from definition of compatible tensor and components of $C$ and $K$.
\end{proof}
\begin{rem}
From \eqref{emt}, we see that the energy momentum tensor of the Vaidya metric \eqref{vm} can be expressed as
$$T = \frac{2 c^4 m'}{8 \pi  G r^2} (\eta\otimes\eta), \ \ \mbox{where } \eta = \{1,0,0,0\}.$$
Now $||\eta|| = 0$ and hence the Vaidya metric is a pure radiation metric \cite{Vaid43}. 
\end{rem}
\begin{rem}
From the value of the local components (presented in Section \ref{com}) of various tensors of the Vaidya metric \eqref{vm}, we can easily conclude that the metric does not fulfill the following geometric structures:\\
(i) super generalized recurrent (\cite{SKA16}, \cite{SRK16}) (ii) weakly symmetric (\cite{SK12}, \cite{TB89}) for $R$, $C$ and $P$, (iii) weakly cyclic Ricci symmetric \cite{SJ06}, (iv) generalized Roter type (\cite{SKgrt}, \cite{SK16}), (v) $R$-space, $C$-space or $P$-space by Venzi \cite{Venz85}, (vi) harmonic or conformal harmonic, (vii) Weyl-pseudosymmetric or Ricci-pseudosymmetric (\cite{D-1991}, \cite{Desz92}, \cite{DGHS11}, \cite{DHV08}).
\end{rem}
\indent From \eqref{sm} and \eqref{vm}, it is clear that the Vaidya metric reduces to Schwarzschild metric \eqref{sm} if $m(u)$ is a non-zero constant. Hence from Theorem \ref{mainthm}, we can conclude the following about the curvature properties of the Schwarzschild metric \eqref{sm}.
\begin{cor}
The Schwarzschild metric \eqref{sm} possesses the following curvature properties:
\begin{enumerate}[label=(\roman*)]
\item The metric is Ricci flat and hence $R = P = W = C = K$.
\item The metric is harmonic, i.e., $div R = 0$.
\item It is Deszcz pseudosymmetric manifold satisfying $R\cdot R = \frac{m}{r^3}Q(g,R)$.
\item The general form of the compatible tensors for $R$ is given by
$$\left
(\begin{array}{cccc}
 a_{_{11}} & a_{_{12}} & 0 & 0 \\
 a_{_{12}} & a_{_{22}} & 0 & 0 \\
 0 & 0 & a_{_{33}} & a_{_{34}} \\
 0 & 0 & a_{_{34}} & a_{_{44}}
\end{array}
\right),$$
where $a_{_{ij}}$ being arbitrary scalars.
\end{enumerate}
\end{cor}
\begin{rem}
From \eqref{emt}, we see that for Schwarzschild metric \eqref{sm}, $T = 0$ and hence the Schwarzschild metric is a vacuum solution of Einstein's field equations.
\end{rem}
\section{\bf Comparisons between Vaidya metric and Ludwig-Edgar pure radiation metric}\label{compar}
In 1997 Ludwig and Edgar \cite{LE97} presented a pure radiation metric which is conformally related to vacuum space or Ricci flat space. The pure radiation metric of Ludwig and Edgar (\cite{LE97}, \cite{SKAA17}) in $(u; r; x; y)$-coordinates ($r>0$ and $x>0$) is given by
\be\label{prm}
ds^2 = \left(x w - p^2 \frac{r^2}{x^2}\right)du^2 + 2 du dr - \frac{4r}{x} du dx - \frac{1}{p^2}(dx^2 + dy^2),
\ee
where $w$ is a smooth function of $u, x, y$, and $p$ is a non-zero constant.\\
\indent Recently, Shaikh et al. \cite{SKAA17} have studied the curvature properties of the Ludwig-Edgar pure radiation metric \eqref{prm}. Since both the Vaidya metric \eqref{vm} and Ludwig-Edgar metric \eqref{prm} represent pure radiation fields, in this section we are mainly interested to make a comparisons between the geometric properties of Vaidya metric \eqref{vm} and Ludwig-Edgar pure radiation metric \eqref{prm}.\\
\newline
%
\textbf{A. Similarities:}\\
(i) Ricci tensors of both are neither Codazzi type nor cyclic parallel but the scalar curvatures are zero.\\
(ii) Both are Ricci simple.\\
(iii) Ricci tensors of both the metrics are Riemann compatible as well as Weyl compatible.\\
(iv) conformal curvature 2-forms of both the metrics are recurrent.\\
\textbf{B. Dissimilarities:}\\
(i) Ludwig-Edgar pure radiation metric is semisymmetric whereas Vaidya metric is of pseudosymmetric Weyl conformal curvature tensor.\\
(ii) Ludwig-Edgar pure radiation metric is weakly Ricci symmetric but Vaidya metric is not weakly Ricci symmetric.\\
(iii) Ludwig-Edgar pure radiation metric is $R$-space by Venzi but Vaidya metric is not $R$-space by Venzi.\\
(iv) Energy momentum tensor of Ludwig-Edgar pure radiation metric is is parallel but energy momentum tensor of Vaidya metric is not parallel.
\section*{\bf Acknowledgment}
The authors are grateful to the referees for the remarks and comments that helped to improve the paper. All the algebraic computations of Section $3$ are performed by a program in Wolfram Mathematica.


\end{document}